\documentclass[11pt, a4paper]{article}

\usepackage{amsthm} 				
\usepackage{amsmath}				
\usepackage{amssymb}				

\usepackage{csquotes}				

\usepackage[T1]{fontenc} 				

\usepackage[a4paper]{geometry}				

\usepackage{tikz-cd}					
\usepackage[all]{xy}					
\usepackage{pgf,tikz}
\usepackage{mathrsfs}
\usetikzlibrary{arrows}

\title{On sign-coherence of $c$-vectors}

\author{Hipolito Treffinger}

\newcommand{\End}{\mbox{End}}
\newcommand{\Hom}{\mbox{Hom}}

\newcommand{\Fac}{\mbox{Fac}}

\newcommand{\F}{\mathcal{F}}
\newcommand{\T}{\mathcal{T}}
\newcommand{\B}{\mathcal{B}}
\newcommand{\Ch}{\mathfrak{C}}

\renewcommand{\v}{\mathsf{v}}

\renewcommand{\mod}{\mbox{mod}}

\newcommand{\rep}[1]{%
  {%
    \tiny%
    \begin{matrix}%
      #1%
    \end{matrix}%
  }%
}

\newcommand{\por}[2]{ \left\langle #1 , #2 \right\rangle }

\newtheorem{theorem}{Theorem}[section]
\newtheorem{corollary}[theorem]{Corollary}
\newtheorem{lemma}[theorem]{Lemma}
\newtheorem{proposition}[theorem]{Proposition}

\theoremstyle{definition}
\newtheorem{definition}[theorem]{Definition}

\theoremstyle{remark}
\newtheorem{remark}[theorem]{Remark}
\newtheorem{example}[theorem]{Example}

\begin{document}

\maketitle



\abstract{Given a finite dimensional algebra $A$ over an algebraically closed field, we consider the $c$-vectors such as defined by Fu in \cite{Fu2017} and we give a new proof of its sign-coherence. Moreover, we characterise the modules whose dimension vectors are $c$-vectors as bricks respecting a functorially finiteness condition.}



\section{Introduction}

Cluster algebras were introduced at the beginning of the century by Fomin and Zelevinsky in \cite{Fomin2001}. 
In the subsequent papers \cite{Fomin2003} and \cite{Fomin2007}, they introduce two families of vectors with integer coefficients that index the cluster variables: the $c$-vectors and the $g$-vectors. 
They have been used to parametrise canonical bases of cluster algebras (see for instance \cite{Musiker2013,Plamondon2013}) and, in the last years, to build the scattering diagrams of cluster algebras by Gross, Hacking, Keel and Kontsevich in \cite{Gross2017}, a powerful tool that was used in order to solve several conjectures on cluster algebras at once and also relate cluster algebras with mirror symmetry. 

From the moment Fomin and Zelevinsky defined the $c$-vectors, they noted that all non-zero entries of a given $c$-vector are either positive or negative and name this phenomenon as \textit{sign-coherence} of $c$-vectors. 
They conjecture that $c$-vectors are always sign-coherent and this was proven true for quivers first by \cite{Derksen2010} and later by Nagao in \cite{Nagao2013}.
The general case of sign-coherence for cluster was proven in \cite{Gross2017}.

Later, cluster algebras started to be categorified using the so-called cluster categories, first introduced by Caldero, Chapoton and Schiffler in \cite{Caldero2005} for the $\mathbb{A}_{n}$ case, and, independently, by Buan, Marsh, Reineke, Reiten and Todorov in \cite{Buan2006a} for acyclic quivers. 
That lead people to give representation theoretic meaning to these families of vectors. 
First, Dehy and Keller introduced in \cite{Deny2008} the set of $g^\dagger$-vectors for 2-Calabi-Yau categories, conjecturally equivalent to $g$-vectors in the corresponding cluster algebra. 
The equivalence follows from the work of Plamondon in \cite{Plamondon2011}.
On the other hand, Nagao proved in \cite{Nagao2013} that one can realise the $c$-vectors of an skew-symmetrizable cluster algebra as a subset of the dimension vectors of functorially finite bricks in the module category of the associated jacobian algebra. The reverse inclusion was proved by N\'ajera-Chavez in \cite{Chavez2015,NajeraChavez2013} for the acyclic and finite case, respectively.

In the recent years, Adachi, Iyama and Reiten introduced in \cite{AIR} the $\tau$-tilting theory. 
This theory succeeds to emulate the combinatorics of cluster algebras on the module category of an algebra, bypassing the construction of a cluster category. 
Therefore, problems that arise naturally in the cluster setting can be stated on this representation theoretic environment, even if the algebra considered does not have an associated cluster algebra.

For instance, of $g$-vectors in this context was given by Adachi, Iyama and Reiten in \cite{AIR}, generalising the approach of \cite{Deny2008}. 
Afterwards, Fu introduced in \cite{Fu2017} the $c$-vectors for finite dimensional algebras as the columns of the inverse of the transpose of matrices of $g$-vectors (see section \ref{sc:Setting} for a rigorous definition). 
Then he showed that $c$-vectors are sign-coherent for every finite dimensional algebra $A$, by showing that $c$-vectors are dimension vectors of certain $A$-modules.
Moreover, he gives an explicit description of the these modules when $A$ is either quasitilted, representation directed or cluster-tilted of finite type. 

In the present paper we take the definition of $c$-vectors given in \cite{Fu2017} and we show that they are the dimension vectors of certain bricks, using the description of stability functions given in \cite{BSTwc}. 
Our first theorem is the following.

\begin{theorem}[Theorem \ref{th:c-matrix}]
Let $A$ be an algebra, $(M,P)$ be a $\tau$-tilting pair in $\emph{mod} A$ and $C_{(M,P)}$ be its $C$-matrix. 
Then there exist a set of bricks $\{B_1, \dots, B_n\}$ such that its dimension vectors are, up to sign, to the columns $\{c_1, \dots, c_n\}$ of $C_{(M,P)}$.
In other words, either $[B_r] = c_r$ or $[B_r] = -c_r$ for all $1 \leq r \leq n$.

In particular the $c$-vectors of $A$ are sign-coherent.
\end{theorem}

This result has two interesting consequences. The first one is that the arrows of the exchange graph of $\tau$-tilting pairs of every algebra can be labeled with $c$-vectors as follows. 

\begin{corollary}[Corollary \ref{cor:labeling}]
Let $A$ be an algebra. Then every arrow in the exchange graph of $\tau$-tilting pairs can be labeled with a positive $c$-vector.
\end{corollary}

Note that the previous labelling coincides with the brick labelling studied for instance in \cite{Demonet2017, BCZ}.

As we said already, $c$-vectors and $g$-vectors are closely related with the scattering diagrams of cluster algebras. 
In \cite{Bridgeland2016a}, Bridgeland proposed a scattering diagram for any finite dimensional algebra and showed that this scattering diagram coincide with the corresponding cluster scattering diagram if the algebra hereditary. 
In \cite{BSTwc} the $g$-vectors were used to give an algebraic description of the wall and chamber structure of an algebra, which happens to be the support of the scattering diagram introduced by Bridgeland. 
As a consequence of our main theorem we prove that the walls surrounding chambers generated by $g$-vectors (this set includes all reachable chambers) are perpendicular to $c$-vectors. For more details on the wall and chamber structure of an algebra see section \ref{sec:wallandchamber}. 
The precise statement is the following.

\begin{corollary}[Corollary \ref{cor:walls}]
Let $(M,P)$ be a $\tau$-tilting pair and $\mathfrak{C}_{(M,P)}$ be the chamber induced by it. Then the walls surrounding $\mathfrak{C}_{(M,P)}$ are defined by the $c$-vectors corresponding to $(M,P)$, that is, the columns of the matrix $C_{(M,P)}$.
\end{corollary}

Now, it is known that the dimension vector of every brick in the module category of an arbitrary algebra is a $c$-vector. 
Our following result gives a classification of the bricks whose dimension vectors are positive $c$-vectors in terms of functorially torsion classes.

\begin{theorem}[Theorem \ref{th:FacM}]
Let $(M,P)$ be a $\tau$-tilting pair and $\B_{(M,P)}^+ :=\{B_1, \dots, B_k\}$ be the set of bricks whose dimension correspond to the positive columns of the $C$-matrix $C_{(M,P)}$ of $(M,P)$.
Then the following holds:
\begin{itemize}
\item $\emph{Hom}_A(B_i, B_j)=0$ whenever $i \neq j$;
\item $\emph{Fac} M$ is the minimal torsion class  containing $\B_{(M,P)}^+$.
\end{itemize}

Moreover, the reciprocal also holds.
Namely, if is $\mathcal{N}$ is a set of bricks in $\emph{mod} A$ such that $\Hom_A(B,B')=0$ for all $B, B' \in \mathcal{N}$ and the minimal torsion class containing $\mathcal{N}$ is functorially finite. 
Then $\mathcal{N}=\B_{(M,P)}^+$ for some $\tau$-tilting pair $(M,P)$.
\end{theorem}

The structure of the article is the following. 
In section \ref{sc:Setting} we give the background material and we establish the setting for the rest of the paper. 
In section \ref{sc:Main} we prove the main results of the paper.
In section \ref{sec:wallandchamber} we investigate how  the results in the previous section improve the description of the wall and chamber structure an algebra given in \cite{BSTwc}.
In section \ref{sc:Example}, we finish the article with the complete analysis of a particular example.

It is important to remark that, independently, Speyer and Thomas (private communication), and J\o rgensen and Yakimov in \cite{JY} proved similar results.
While preparing this second version the author was informed that Asai also proved similar results in \cite{Asai2016}.

\paragraph{Acknowledgements}
The author thankfully acknowledges Peter J\o rgensen for the exchanges that lead to the work in this paper. 
He also is grateful to Kiyoshi Igusa, Hongwei Niu and Sibylle Schroll for the useful discussions. 
He does not want to forget the valuable comments of Bernhard Keller and the anonymous referee on a previous version of the paper that lead to great improvement of this article.
This project was supported by the EPSRC founded project EP/P016294/1.



\section{Setting}\label{sc:Setting}

In this paper $A$ is a finite dimensional algebra over an algebraically closed field $k$. 
By $\mod A$ we mean the category of finite generated right $A$-modules and $\tau$ represents the Auslander-Reiten translation in $\mod A$. 
If one considers $A$ as a module over itself, then $A$ can be written as $A=\bigoplus_{i=1}^{n} P(i)$, where $P(i)$ is the $i$-th indecomposable projective $A$-module. 

The Grothendieck group of $A$ is noted $K_{0}(A)$, where rk$(K_{0}(A))=n$. 
It is known that $K_{0}(A)$ is isomorphic to $\mathbb{Z}^{n}$. 
In this paper we consider the embedding $\Delta: K_{0}(A)\to \mathbb{R}^{n}$ of $K_{0}(A)$ to $\mathbb{R}^{n}$ given by $\Delta(M)=[M]$ where $[M]$ is the dimension vector of $M$, for every $A$-module $M$. If the context is unambiguous we simply say that $[M]$ is a vector of $\mathbb{R}^{n}$.

Given an $A$-module $M$, we denote by $|M|$ the number of non-isomorphic indecomposable direct summands of $M$.

When we write $\langle -,-\rangle$ we are referring to the canonical inner product in $\mathbb{R}^n$ which is defined by
$$\langle v,w\rangle=\sum_{i=1}^n v_iw_i$$
for every $v$ and $w$ in $\mathbb{R}^n$.

We say that an $A$-module $M$ is a \textit{brick} if its endomorphism algebra $\End_{A}(M)$ is a division ring. 

\subsection{$\tau$-tilting theory}

The $\tau$-tilting theory was introduced by Adachi, Iyama and Reiten in \cite{AIR}. This theory provides a framework to study problems arising in cluster algebras in the module category of an arbitrary algebra. 
In the proofs of this paper the $\tau$-rigid ($\tau$-tilting) pairs play a central role. 
They are defined as follows. 

	\begin{definition}\cite[Definition 0.1 and 0.3]{AIR}\label{list}
		Let $A$ an algebra, $M$ an $A$-module and $P$ a projective $A$-module. The pair $(M,P)$ is said \textit{$\tau$-rigid} if:

		\begin{itemize}
			\item $\Hom_{A}(M,\tau M)=0$;
			\item $\Hom_{A}(P,M)=0$.
		\end{itemize}
		Moreover, we say that $(M,P)$ is \textit{$\tau$-tilting} (or \textit{almost $\tau$-tilting}) if $|M|+|P|=n$ (or $|M|+|P|=n-1$, respectively).
	\end{definition}

From now on, when we say that $(M,P)$ is a $\tau$-tilting ($\tau$-rigid) pair, we are assuming that $M$ and $P$ are basic and its decomposition as direct sum of indecomposable modules can be written as $M=\bigoplus_{i=1}^kM_i$ and $P=\bigoplus_{j=k+1}^nP_j$ ($M=\bigoplus_{i=1}^kM_i$ and $P=\bigoplus_{j=k+1}^t P_j$, with $t\leq n$, respectively).

The $\tau$-tilting is a generalisation of classical tilting theory which is capable of describe all the functorially finite torsion classes in the module category of an algebra in terms of $\tau$-tilting pairs. 

\begin{theorem}\cite[Theorem 2.7]{AIR}\cite[Theorem 5.10]{Auslander1981}
There is a well defined function $\Phi: \mathrm{s\tau\text{-}rig} A \to \mathrm{f\text{-}tors}$ from $\tau$-rigid pairs to functorially finite torsion classes given by 
$$\Phi(M,P)=\emph{Fac} M:=\{X \in\emph{mod} A : M^{n} \to X \to 0 \text{ for some $n\in\mathbb{N}$}\}.$$
Moreover, $\Phi$ is a bijection if we restrict it to $\tau$-tilting pairs.
\end{theorem}

Recall that a torsion pair $(\T, \F)$ in $\mod A$ is a pair $\T$ and $\F$ of full subcategories in $\mod A$ such that $\Hom_A (X,Y)=0$ for every $X \in \T$ and $Y \in \F$ and which are maximal with respect to this property.
Given a torsion pair $(\T, \F)$, we say that $\T$ is a torsion class and $\F$ is a torsion free class. 
Moreover, $\T$ is a full subcategory of $\mod A$ closed under quotients and extensions, while $\F$ is a full subcategory of $\mod A$ closed under submodules and extensions.
It is well know that for every subcategory $\T$ closed under quotients and extensions there exists a torsion free class $\mathcal{F}$ such that $(\mathcal{T}, \mathcal{F})$ is a torsion pair. 
In particular, if $\mathcal{T}=\Fac M$ for some $A$-module $M$, then $\mathcal{F} = M^{\perp} = \left\{ N \in \mod A : \Hom_{A}(M,N)=0 \right\}$.

In this paper, given a module $M$, we denote by $T(M)$ the minimal torsion class containing $M$.

Recall that for every torsion pair $(\mathcal{T},\mathcal{F})$ and every module $M$ in mod$A$ there exists the a canonical short exact sequence
$$0\to tM\to M\to M/tM\to 0$$
where $tM\in\mathcal{T}$ and $M/tM\in\mathcal{F}$, which is unique up to isomorphism.

\subsection{Integer vectors associated to modules}

In this paper we denote by $K_{0}(A)$ the Grothendieck group of $A$. By abuse of notation, for every $M$ in $\mod A$ we identify its equivalence class in $K_{0}(A)$ with its dimension vector, and we denote it by $[M]$. 

For each cluster algebra, one can associate to it two different and complementary set of vectors: the $c$-vectors and $g$-vectors (see \cite{Fomin2007}). 
In this paper we study their representation theoretic versions. 
The definition of $g$-vectors is the following. 

\begin{definition}\cite[Section 5]{AIR}
Let $A$ be an algebra, $M$ be an $A$-module and $$P_1\longrightarrow P_0\longrightarrow M\longrightarrow 0$$ be the minimal projective presentation of $M$, where $P_0=\bigoplus\limits_{i=1}^n P(i)^{a_i}$ and $P_1=\bigoplus\limits_{i=1}^n P(i)^{a'_i}$. 
The $g$-vector of $M$ is set to be
$$g^M=(a_1-a'_1, a_2-a'_2,\dots, a_n-a'_n).$$
\end{definition}

There is important homological information that arises when $g$-vectors and dimension vectors interplay, as showed by Auslander and Reiten in \cite{Auslander1985}.

\begin{theorem}\cite[Theorem 1.4.(a)]{Auslander1985}\label{formula}
Let $M$ and $N$ be arbitrary modules over an algebra $A$. Then we have
$$\langle g^{M},[N]\rangle=\dim_k(\emph{Hom}_A(M,N))-\dim_k(\emph{Hom}_A(N,\tau_A M)).$$
\end{theorem}

In particular, if we restrict ourselves to $\tau$-tilting pairs, $g$-vectors have other emerging features. 
See for instance the following result proved by Adachi, Iyama and Reiten in \cite{AIR}.

\begin{theorem}\cite[Theorem 5.1]{AIR}\label{basis}
Let $(M,P)$ be a $\tau$-tilting pair. Then the set 
$$\{g^{M_{1}},\dots, g^{M_{k}},-g^{P_{k+1}},\dots, -g^{P_{n}}\}$$
form a basis of $\mathbb{Z}^{n}$.
\end{theorem}

This important property of $g$-vectors allowed Fu to define $c$-vectors as follows.

\begin{definition}\cite[Subsection 3.1]{Fu2017}\label{def:C-mat}
Let $(M,P)$ be a $\tau$-tilting pair and consider the set 
$$\{g^{M_{1}},\dots, g^{M_{k}},-g^{P_{k+1}},\dots, -g^{P_{n}}\}$$
of $g$-vectors corresponding to the indecomposable direct factors of $M$ and $P$. 
Define the $g$-matrix $G_{(M,P)}$ of $(M,P)$ as 
$$G_{(M,P)}=\left(g^{M_{1}},\dots, g^{M_{k}},-g^{P_{k+1}},\dots, -g^{P_{n}}\right)=\left( \begin{matrix}  
				(g^{M_{1}})_{1} & \dots   & (-g^{P_{n}})_{1} \\
				\vdots      & \ddots & \vdots     \\
				(g^{M_{1}})_{n}  & \dots   & (-g^{P_{n}})_{n}
		\end{matrix} \right).$$
Then the $c$-matrix of $(M,P)$ is defined as
$$C_{(M,P)}=(G^{-1}_{(M,P)})^{T}.$$ 
Also, each column of $C_{(M,P)}$ is called a \textit{$c$-vector} of $A$. Moreover, we denote by $\mathsf{cv}(A)$ the set of all $c$-vectors of $A$. 
\end{definition}

\begin{remark}
Note that, since $\mod A$ is a Krull-Schmidt category, the definition of $G_{(M,P)}$ for every $\tau$-tilting pair is unique up to permutation of columns.
Therefore, $C_{(M,P)}$ is also uniquely defined up to permutation of columns.
\end{remark}

In this paper we are interested in the property of sign-coherence of $c$-vectors. The formal definition is the following.

\begin{definition}\cite{Fomin2003}
A vector $\v=(\v_{1}, \dots, \v_{n})$ is said to be \textit{positive} if $\v_{i} \geq 0$ for every $1\leq i\leq n$. Dually, $\v$ is said to be \textit{negative} if $\v_{i}\leq 0$ for every $1\leq i \leq n$. Finally, a set $\mathcal{S}$ of vectors is said to be \textit{sign-coherent} if $\v$ is either positive or negative for every $\v \in \mathcal{S}$.
\end{definition}

The sign-coherence of $c$-vectors was first formulated as a conjecture in the cluster setting by Fomin and Zelevinsky in \cite{Fomin2007}. 
This conjecture was proven to be true for quiver cluster algebras by Derksen, Weymann and Zelevinsky in \cite{Derksen2010}.
Later, Fu introduced $c$-vectors for every finite dimensional algebra in \cite{Fu2017} and prove the following.

\begin{theorem}\cite[Theorem 3.1]{Fu2017}
Let $A$ be an algebra. Then the set of $c$-vectors of $A$ is sign-coherent.
\end{theorem}

We denote the set of positive and negative $c$-vectors by $\mathsf{cv}^{+}(A)$ and $\mathsf{cv}^{-}(A)$, respectively.

\subsection{Stability conditions}\label{ssc:stability}

We will attack the problem of sign-coherence of $c$-vectors from the point of view of stability conditions. 
The definition of $\theta$-semistables modules was introduced by King in \cite{KING1994} in terms of linear functionals. 
Given that we intend to have a geometrical realisation of these objects, we adapt the definition in terms of vectors as follows.

\begin{definition}\cite[Definition 1.1]{KING1994}\label{stability}

Let $\theta$ be a vector of $\mathbb{R}^n$, the dual vector space of $\mathbb{R}^n$. An $A$-module $M\in\mod A$ is called \textit{$\theta$-stable} (or \textit{$\theta$-semistable}) if $\langle \theta, [M]\rangle=0$ and $\langle \theta, [L]\rangle<0$ ($\langle \theta, [L]\rangle\leq 0$, respectively) for every proper submodule $L$ of $M$. 
\end{definition}

Given a vector $\theta\in \mathbb{R}^n$, we denote by $\mod A^{ss}_{\theta}$ the category of $\theta$-semistable modules. 
It is known that $\mod A^{ss}_{\theta}$ is an abelian category for every $\theta$, where $\theta$-stable modules correspond to the simple objects in $\mod A_{\theta}^{ss}$. 
Note that a $\theta$-stable module is necessarily a brick in $\mod A$ by \cite[Theorem 1]{Rudakov1997} and \cite[Proposition 3.4]{Rudakov1997}. 

Sometimes, a more precise description of $\mod A^{ss}_{\theta}$ can be given. 
For instance, if $\theta$ can be defined using the $g$-vectors of a $\tau$-rigid pair $(M,P)$, such a description was given in \cite{BSTwc}. 
The result is the following.

\begin{theorem}\label{stablemodcat}\cite[Theorem 3.14]{BSTwc}
Let $(M,P)$ be a $\tau$-rigid pair and consider $\theta_{\alpha(M,P)}\in\mathbb{R}^n$ defined as
$$\theta_{\alpha(M,P)}:= \sum_{i=1}^{k} \alpha_i g^{M_{i}} - \sum_{j=k+1}^{t} \alpha_j g^{P_{j}},$$ 
where $M=\bigoplus_{i=1}^kM_i$, $P=\bigoplus_{j=k+1}^tP_j$ and  $\alpha_l > 0$ for every $1 \leq l \leq t$. 
Then there exists an algebra $\tilde{A}_{(M,P)}$ which is independent of $\alpha$ such that there is an equivalence of categories 
$$F: \emph{mod} A^{ss}_{\theta_{\alpha(M,P)}} \longrightarrow \emph{mod} \tilde{A}_{(M,P)}$$
from category $\emph{mod} A^{ss}_{\theta_{\alpha(M,P)}}$ of $\theta_{\alpha(M,P)}$-semistable modules to $\emph{mod} \tilde{A}_{(M,P)}$, the module category of an algebra $\tilde{A}_{(M,P)}$, with quasi-inverse 
$$G:\emph{mod} \tilde{A}_{(M,P)} \to \emph{mod} A^{ss}_{\theta_{\alpha(M,P)}}.$$ 
In particular, there are exactly $rk(K_0(A))-t$ non-isomorphic $\theta_{\alpha(M,P)}$-stable modules.
\end{theorem}

\begin{remark}
The algebra $\tilde{A}_{(M,P)}$ in the previous theorem correspond to the algebra arising from $\tau$-tilting reduction developed by Jasso in \cite{Jasso2015}.
The explicit construction of the algebra can be found in the original article. 
Also the main ideas are given in \cite[Section 2]{BSTwc}.
\end{remark}

Using the previous result, Br\"ustle, Smith and Treffinger constructed explicitly one $\theta_{\alpha(M,P)}$-stable module when $(M,P)$ is an almost $\tau$-tilting pair. 

\begin{proposition}\cite[Proposition 3.17]{BSTwc}\label{prop:generator}
Let $(M,P)$ be an almost $\tau$-tilting pair and $\theta_{\alpha(M,P)}$ be as before. 
Then there is a $\theta_{\alpha(M,P)}$-semistable module $N$ that is constructed as follows: 

Let $(M_1,P_1)$ and $(M_2,P_2)$ be the two $\tau$-tilting pairs containing $M$ and $P$ as a direct factor. Order them such that $\emph{Fac} M_1 \subset \emph{Fac} M_2$.
Then $N$ is the cokernel of the right \emph{add}$M$-approximation of $M_2$. 
\end{proposition}



\section{$c$-vectors as dimension vectors of bricks}\label{sc:Main}

This section is organised as follows.
In the first subsection, using theorem \ref{stablemodcat}, we find $n$ bricks $\{B_1, \dots, B_n \}$ which are naturally associated to a given $\tau$-tilting pair $(M,P)$.
Then we build a matrix $X_{(M,P)}$ using of dimension vectors of those bricks, which will play a key role in the rest of the paper.

In the second we show that one can always build the matrix $C_{(M,P)}$ from $X_{(M,P)}$. 
As a corollary it follows that every positive (negative) $c$-vectors is the (opposite of the) dimension vector of a brick in $\mod A$.
Hence, the sign coherence of $c$-vectors follows at once.

In the final subsection we study the relation between $c$-vectors and torsion classes.
More precisely, we show that the modules whose dimension vectors are positive the $c$-vectors of $C_{(M,P)}$ are the minimal set of bricks in $\mod A$ generating the torsion class $\Fac M$.
We also show the reciprocal, that is, we show that the dimension vectors of every minimal set of bricks in $\mod A$ generating a functorially finite torsion class are the positive columns of the $C$-matrix $C_{(M,P)}$ for a certain $\tau$-tilting pair $(M,P)$.

\subsection{Bricks associated to a $\tau$-tilting pair}\label{ssc:bricks}

Let $(M,P)$ be a $\tau$-tilting pair. 
Then their decomposition as sum of indecomposable direct summands is the following.
$$M = \bigoplus_{i=1}^k M_i \text{ \qquad\qquad\qquad }P = \bigoplus_{j=k+1}^n P_j$$
We can suppose without loss of generality the that the indecomposable direct summands of $M$ and $P$ are pairwise non-isomorphic, i.e., $M$ and $P$ are basic.

Now, given $(M,P)$, one can construct for every $1\leq r\leq n$ an almost $\tau$-tilting pair $(M,P)_{r}$ as follows.
$$(M,P)_{r}=\left(\bigoplus_{i\neq r} M_{i}, \bigoplus_{j\neq r} P_{j} \right)$$

Then, define for every $r$ the vector $\theta_r$ in the following way.
$$\theta_r:= \sum_{i\neq r} g^{M_i} - \sum_{j\neq r} g^{P_j}$$

Hence, theorem \ref{stablemodcat} implies that for every $r$ there exists a brick $B_r$ which is $\theta_r$-stable.
Moreover, $B_r$ is unique up to isomorphism. 
Therefore, for every $\tau$-tilting pair we have a set $\mathcal{B}_{(M,P)}$ of $\theta_r$-stable modules for every $r$ between $1$ and $n$. 
Formally:
$$\mathcal{B}_{(M,P)}:= \{B_r : B_r \text{ is $\theta_r$-stable for $1 \leq r \leq n$ }\}.$$

Now we are ready to define $X_{(M,P)}$ as the square matrix having as $r$-th column the dimension vector $[B_r]$ of the brick $B_r \in \mathcal{B}_{(M,P)}$. 

$$X_{(M,P)}:=\left(\begin{matrix}
				[B_{1}] &|& [B_2] &|& \dots   &|& [B_{n}]
		\end{matrix}\right)=\left(\begin{matrix}
				[B_{1}]_{1} & \dots   & [B_{n}]_{1} \\
				\vdots      & \ddots & \vdots       \\
				[B_{1}]_{n}  & \dots   & [B_{n}]_{n}
		\end{matrix}\right)
$$
If one takes the set of all $g$-vectors and the set of all bricks in the module category that can be obtained as we just did, there is no well defined function between these two sets. 
However, if we restrict ourselves to a particular $\tau$-tilting pair the situation changes, since the previous construction gives a bijection between the $g$-vectors of the indecomposable direct summands of $(M,P)$ and the elements of $\mathcal{B}_{(M,P)}$.
Therefore, sometimes we may fall in an abuse of language saying that a the brick $B_r\in \mathcal{B}_{(M,P)}$ is the brick \textit{associated} or \textit{corresponding} to the $g$-vector $g^{M_r}$ if $1\leq r \leq k$ or $-g^{P_r}$ if $k+1 \leq r \leq n$.

\subsection{The relation between $C_{(M,P)}$ and $X_{(M,P)}$}

Fix a $\tau$-tilting pair $(M,P)$.
As is said in the title, this subsection is devoted to show the relation between the $c$-matrix $C_{(M,P)}$ defined in definition \ref{def:C-mat} and the matrix $X_{(M,P)}$ introduced in subsection \ref{ssc:bricks}. 
The explicit relation is given by the following theorem.

\begin{theorem}[Theorem \ref{th:c-matrix}]
Let $(M,P)$ be a $\tau$-tilting pair.
Then there exists a diagonal matrix $D$ which is invertible over $\mathbb{Z}$ and such that 
$$C_{(M,P)}= X_{(M,P)} D.$$
\end{theorem}

Remark that the fact that $D$ is diagonal and invertible over $\mathbb{Z}$ implies that every in the diagonal is either $1$ or $-1$.
We start with the following proposition that will take care of the positive entries of $D$.

\begin{proposition}\label{prop:SinFac}
Let $(M,P)$ be a $\tau$-tilting pair, $(\emph{Fac} M, M^{\perp})$ be the torsion pair induced by $(M,P)$ and let $N$ be a $\theta_r$-semistable module. 
If $\por{\theta_{(M,P)}}{N} > 0$ then $N$ belongs to $\emph{Fac} M$.
\end{proposition}

\begin{proof}
Let $N_r$ the $\theta_r$-semistable module of proposition \ref{prop:generator}.
It follows from $N_{r}$ is a generator of $\mod A^{ss}_{\theta_{(M,P)_{r}}}$ by \cite[Theorem 3.15]{Jasso2015}.
Moreover \cite[Proposition 3.17]{BSTwc} implies that $N_{r}$ belongs to $\Fac M$. 
Therefore any other $\theta_{r}$-semistable module $N$ belong to $\Fac M$ because $\Fac M$ is a torsion class.
\end{proof}

The following lemma is key to show that $c$-vectors are the dimension vectors of certain modules.

\begin{lemma}\label{lem:stables}
Let $(M,P)$ be a $\tau$-tilting pair, $B_r \in \mathcal{B}_{(M,P)}$ be as defined in subsection \ref{ssc:bricks} and  
$$\theta_{(M,P)}:= \sum_{i=1}^{k} g^{M_{i}} - \sum_{j=k+1}^{t} g^{P_{j}}.$$  
If $$\langle \theta_{(M,P)} , [B_{r}]\rangle > 0$$, then $$\langle \theta_{(M,P)} , [B_{r}]\rangle = 1.$$
\end{lemma}

\begin{proof}
Let $B_r \in \mathcal{B}_{(M,P)}$. 
As a first remark, one can see that theorem \ref{stablemodcat} implies that $\langle \theta_{(M,P)} , [B_r] \rangle \neq 0$ because $(M,P)$ is a $\tau$-tilting module. 

A second easy remark is that $\theta_{(M,P)}=\theta_{(M,P)_r} + g^{M_r}$ if $1 \leq r \leq k$ or $\theta_{(M,P)}=\theta_{(M,P)_r} + g^{P_r}$ if $k+1 \leq r \leq n$.

Suppose that $k+1 \leq r \leq n$.
Then the linearity of the inner product implies that 
	\begin{align*}
		\langle \theta_{(M,P)} , [B_{r}] \rangle 	&= \left\langle \theta_r - g^{P_{r}},  [B_{r}] \right\rangle \\
							&= \langle \theta_r, [B_r] \rangle + \langle -g^{P_r} , [B_r] \rangle
	\end{align*}
By construction we have that $B_r$ is $\theta_r$-stable. 
This implies in particular that $\langle \theta_r , [B_r] \rangle = 0$.
Hence 
\begin{equation*}
\por{\theta_{(M,P)}}{[B_r]} = \por{-g^{P_r}}{[B_r]} =  -\dim_{k}(\Hom_{A}(P_{r}, B_{r}))\leq 0,
\end{equation*}
a contradiction with our hypothesis. 
Then we have that $1\leq r \leq k$.
In that case 
\begin{equation*}
\por{\theta_{(M,P)}}{[B_r]} = \por{g^{M_r}}{[B_r]} =  \dim_{k}(\Hom_{A}(P_{r}, B_{r})) - \dim_k(\Hom_A (B_r, \tau M_r) ).
\end{equation*}

Since $B_r$ is $\theta_r$-stable we have that $B_r \in \Fac M$ by proposition \ref{prop:SinFac}.
Hence $$\Hom_{A}( B_{r} , \tau_{A} M_{r})=0.$$ 
So,
 	\begin{align*}
 		\theta_{(M,P)}([B_{r}])	&= \dim_{k}(\Hom_{A}(M_{r}, B_{r})).
 	\end{align*}
Therefore, is enough to show that $\dim_{k}(\Hom_{A}(M_{r}, B_{r}))=1$ to complete the proof.

Let $(M,P)_r$ be the almost $\tau$-tilting pair as in subsection \ref{ssc:bricks}, 
$$( \T_r, \F_r ) = \left( \Fac \left(\bigoplus_{i\neq r} M_i \right), \left(\bigoplus_{i\neq r} M_i \right)^{\perp} \right)$$ 
the torsion pair associated to it and let
$$0\to tM_r \overset{l}\to M_r \overset{p}\to M_r/tM_r\to 0$$
be the canonical short exact sequence with respect to $(\T_r, \F_r)$.
	
We have that $N_r=M_r/tM_r$ by \cite[Lemma 2.3]{BSTwc}. 
Hence $N_r$ is an Ext-projective module in $\mod A^{ss}_{\theta_{(M,P)_{r}}}$ by \cite[Theorem 3.15]{Jasso2015}. 
Therefore Theorem \ref{stablemodcat} implies that $N_r=\bigoplus\limits_{i=1}\limits^{t}(G(\tilde{A}_{(M,P)_r}))$ for some natural number $t$. Consider, for every $1\leq j\leq t$, the following commutative diagram 
 $$
 \xymatrix{
 		0\ar[r]	&tM_r\ar[r]\ar@{=}[d]	&E_j\ar^{q_{j}}[r]\ar^{f_j}[d]	&G(\tilde{A}_{(M,P)_r})\ar[r]\ar^{\iota_j}[d]	&0 	\\
		0\ar[r]	&tM_r\ar^{l}[r]			&M_r\ar^{p}[r]				&N_r\ar[r]				&0}
 $$
where $E_{j}$ is the pullback of $p$ and the canonical monomorphism $\iota_{j}$ for every $j$. 
Then $E_j\in\Fac M$ for every $1\leq j\leq t$ because $\Fac M$ is closed under extensions. Consequently, so does $\bigoplus_{j=1}^t E_j$.
 
Consider the morphisms 
$$f=[f_1, \dots, f_t]:\bigoplus_{j=1}^t E_j \to M_r ,$$  
$$q=\bigoplus_{j=1}^{t} q_{j}:\bigoplus_{j=1}^t E_j \to \bigoplus_{j=1}^{t} G(\tilde{A}_{(M,P)_r}),$$ 
$$id^{t}=[id, \dots, id]: \bigoplus_{j=1}^{t} tM_{r} \to tM_{r}$$
and 
$$\iota=[\iota_{1}, \dots, \iota_{t}]: \bigoplus_{j=1}^{t} G(\tilde{A}_{(M,P)_r}) \to N_{r}.$$ 
Then we construct the following commutative diagram.
$$
 \xymatrix{
 		0\ar[r]	&\bigoplus_{j=1}^{t} (tM_r)\ar[r]\ar^{id^{t}}[d]	&\bigoplus_{j=1}^tE_j\ar^{q}[r]\ar^{f}[d]	&\bigoplus_{i=1}^{t} G(\tilde{A}_{(M,P)_r})\ar[r]\ar^{\iota}[d]		&0 	\\
		0\ar[r]	&tM_r\ar^{l}[r]							&M_r\ar^{p}[r]						&N_r\ar[r]								&0}
 $$
Note that, by hypothesis, $\iota$ is an isomorphism and, by construction, $id^{t}$ is an epimorphism. 
Therefore, we can apply the snake lemma to show that $f$ is an epimorphism. 
Since $M_r$ is an Ext-projective module for $\Fac M$ and $\bigoplus_{j=1}^t E_j\in\Fac M$, we deduce that $f$ must split. 
So, $M_{r}$ is an indecomposable direct factor of $\bigoplus_{j=1}^t E_j\in\Fac M$. 
Then, $M_{r}$ is a direct factor of some $E_{j}$.

On the other hand, we know that $\dim_{k}(\Hom_{A}(M_{r},B_{r})) \neq 0$ and 
	\begin{align*}
		\dim_{k}(\Hom_{A}(M_{r}, B_{r})) 	&	\leq \dim_{k}(\Hom_{A}(E_{j}, B_{r}))		\\
									&	= \dim_{k}(\Hom_{A}(G(\tilde{A}_{(M,P)_r}), B_{r})).		
	\end{align*}
Moreover, $B_{r}$ is, by hypothesis, the image over the functor $G : \mod \tilde{A}_{(M,P)_r} \to \mod A$ of a representative of the unique isomorphism class of the simple $\tilde{A}_{(M,P)_r}$-modules. 
Therefore 
$$\dim_{k}(\Hom_{A}(M_{r}, B_{r})) \leq \dim_{k}(\Hom_{\tilde{A}_{(M,P)_r}}(\tilde{A}_{(M,P)_r},S))=1$$ 
Then we can conclude that $\dim_{A}(\Hom_{A}(M_{r},B_{r}))=1$ as claimed.
\end{proof}

Now we need to consider the cases of bricks such that $\por{\theta_r}{[B_r]}<0$.
However, in this case the inequality could come from the fact that either $\Hom_A(B_r, \tau M_r) \neq 0$ or $\Hom_A(P_r, M_r) \neq 0$. 
Hence, using dual-like arguments is not enough.
Instead, we will use some basic linear algebra. 

Let consider a basis $B=\{v_1, \dots, v_n\}$ of $\mathbb{Z}^n$ and let $B^*=\{v_1^*, \dots, v_n^*\}$ be its dual basis, i.e., a set of linearly independent vectors in $\mathbb{Z}^n$ such that $\por{v_i}{v_j^*}=1$ if $i=j$ and $\por{v_i}{v_j^*}=0$ otherwise. 
Now consider a new basis $B'=\{v'_1, \dots, v'_n\}$ of $\mathbb{Z}^n$ such that such that $v_1 \neq v'_1$ and $v_i=v'_i$ for all $2\leq i\leq n$.
In general the dual basis $(B')^*$ of $B'$ can be quite different of the dual basis $B^*$ of $B$. 
However something can said, as shown in the following result. 
We include the proof for the convenience of the reader.

\begin{lemma}\label{lem:minus}
Let $B=\{v_1, \dots, v_n\}$ and $B'=\{v'_1, \dots, v'_n\}$ two basis of $\mathbb{Z}^n$ as in the previous paragraph having dual basis $B^*=\{v^*_1, \dots, v^*_n\}$ and $(B')^*=\{(v'_1)^*, \dots, (v'_n)^*\}$, respectively.
Then $\por{v_1}{(v'_1)^*}=-1$. In particular, $(v'_1)^*=-v_1^*$.
\end{lemma}

\begin{proof}
Given that $B^*$ is a basis of $\mathbb{Z}^n$, we have that $(v'_1)$ can be written as a linear combination of the elements of $B^*$.
Moreover, given that $B^*$ is a dual basis of $B$, this linear combination has the following form.
$$(v'_1)^* = \sum_{i=1}^n \por{v_i}{(v'_1)^*} v_i^* $$
By hypothesis, we have that $\por{v_i}{(v'_1)^*}=0 $ if $i\neq 1$.
Therefore 
$$(v'_1)^* = \por{v_1}{(v'_1)^*} v_i^*.$$
Using identical arguments one can prove that 
$$v_1^* = \por{v'_1}{v_1^*} (v'_i)^*.$$
Hence 
$$v_1^* = \por{v'_1}{v_1^*} (v'_i)^* = \por{v'_1}{v_1^*}\por{v_1}{(v'_1)^*} v_i^*,$$
implying that is invertible in $\mathbb{Z}$.
Since $B$ and $B'$ are different basis, we have that $\por{v_1}{(v'_1)^*} \neq 1$.
So, the equality $\por{v_1}{(v'_1)^*}=-1$ must hold.
This implies that 
$$(v'_1)^* = - v_i^*,$$
as claimed.
This finishes the proof.
\end{proof}

Now we are able to prove the main theorem of this subsection, which generalises \cite[Lemma 5.14 and 5.15]{Gross2017} and \cite[Theorem 4.3.1]{Igusa2009}.  

\begin{theorem}\label{th:c-matrix}
Let $(M,P)$ be a $\tau$-tilting pair.
Then there exists a diagonal matrix $D$ which is invertible and such that 
$$C_{(M,P)}=  X_{(M,P)} D.$$
\end{theorem}

\begin{proof}
Let $(M,P)$ be a $\tau$-tilting pair and let $\{g^{M_1}, \dots, g^{M_k}, -g^{P_{k+1}}, \dots, -g^{P_{n}}\}$ be the set of $g$-vectors of its indecomposable direct summands and be $\{S_{1}, \dots, S_{n}\}$ as in Subsection \ref{ssc:stability}. 
For the sake of simplicity we denote $\{g^{M_1}, \dots, g^{M_r}, -g^{P_{k+1}}, \dots, -g^{P_{n}}\}$ simply by $\{g^{1}, g^{2}, \dots, g^{n}\}$. 

Consider $G_{(M,P)}^{T}$, the transpose of the $g$-matrix of the $\tau$-tilting pair $(M,P)$, and multiply it to the right with the matrix $X_{(M,P)}$ constructed in subsection \ref{ssc:bricks}.

$$
G_{(M,P)}^{T} X_{(M,P)}=\left( \begin{matrix}  
				(g^{1})_{1} & \dots   & (g^{1})_{n} \\
				\vdots      & \ddots & \vdots     \\
				(g^{n})_{1}  & \dots   & (g^{n})_{n}
		\end{matrix} \right) \left(\begin{matrix}
				[B_{1}]_{1} & \dots   & [B_{n}]_{1} \\
				\vdots      & \ddots & \vdots       \\
				[B_{1}]_{n}  & \dots   & [B_{n}]_{n}
		\end{matrix}\right). 
$$

Note that, by definition, 
$$\por{\theta_r}{[B_r]}=\left\langle \sum_{i \neq r}  g^{M_{i}} - \sum_{j\neq r}  g^{P_{j}},  [S_{r}] \right\rangle = 0.$$ 
Moreover, 
$$\por{\theta_{(M,P)}}{[B_r]} \neq 0$$
by Theorem \ref{stablemodcat}. 
Then, the linearity of $\theta_{(M,P)}$ implies that 
$$\langle g^{i}, [S_{j}]\rangle \neq0$$
if and only if $i = j$.
That implies that $G_{(M,P)}^{T} X_{(M,P)} = D$, where $D$ is diagonal.
$$
G_{(M,P)}^{T} X_{(M,P)}=\left( \begin{matrix}  
				(g^{1})_{1} & \dots   & (g^{1})_{n} \\
				\vdots      & \ddots & \vdots     \\
				(g^{n})_{1}  & \dots   & (g^{n})_{n}
		\end{matrix} \right) \left(\begin{matrix}
				[B_{1}]_{1} & \dots   & [B_{n}]_{1} \\
				\vdots      & \ddots & \vdots       \\
				[B_{1}]_{n}  & \dots   & [B_{n}]_{n}
		\end{matrix}\right) = \left(\begin{matrix}
				\lambda_{1} & \dots   & 0 \\
				\vdots           & \ddots & \vdots     \\
				0                   & \dots   & \lambda_{n}
		\end{matrix} \right)=D
$$
Moreover, lemma \ref{lem:stables} implies that if $\lambda_r > 0$, then $\lambda_r = 1$.

Suppose now that $\lambda_r < 0$ and consider the almost $\tau$-tilting pair $(M,P)_r$ defined in subsection \ref{ssc:bricks}. 
In that case \cite[Theorem 2.18]{AIR} implies that there exist a $\tau$-tilting pair $(\tilde{M}, \tilde{P})$ different from $(M,P)$ completing $(M,P)_r$.
Therefore the $g$-vectors $\{\tilde{g}^1, \dots, \tilde{g}^n\}$ of the indecomposable direct summands of $(\tilde{M}, \tilde{P})$ form a basis of $\mathbb{Z}^n$ such that $\tilde{g}^i=g^i$ if and only if $i \neq r$.
Moreover, \cite[Theorem 2.18]{AIR} implies that $[B_r]\in \Fac \tilde{M}$. 
Hence lemma \ref{lem:stables} implies that $\por{\tilde{g}^r}{[B_r]}=1$. 
So, we can apply lemma \ref{lem:minus} to conclude that $\por{g^r}{[B_r]}=-1$.

All this argument shows that $D$ is a diagonal matrix having only $1$ or $-1$ in the diagonal. 
Moreover, is easy to see that $D^2$ is the identity matrix. 
Therefore, multiplying by $D$ on the right we find the following equality. 
$$G_{(M,P)}^{T} X_{(M,P)} D = D^2 = Id$$
Hence, definition \ref{def:C-mat} implies $X_{(M,P)}D=C_{(M,P)}$, finishing the proof.
\end{proof}

As an immediate consequence of the previous result we get the following corollaries.

\begin{corollary}
Let $A$ be an algebra. Then the $c$-vectors of $A$ are sign-coherent.
\end{corollary}

\begin{proof}
In theorem \ref{th:c-matrix} is shown that every $c$-vector is either the dimension vector of a module or its opposite. 
Hence $c$-vectors are sign-coherent.
\end{proof}

\begin{corollary}
Let $A$ be an algebra. 
Then every positive $c$-vector is the opposite of a negative $c$-vector and, reciprocally, every negative $c$-vector is the opposite of a positive $c$-vector.
$\qed$
\end{corollary}

\begin{corollary}\label{cor:c-dim}
Let $A$ be an algebra. 
Then every positive $c$-vector is the dimension vector of a brick.
$\qed$
\end{corollary}

In previous work by N\'ajera-Ch\'avez (see \cite[Theorem 6]{NajeraChavez2013} and \cite[Theorem 11]{Chavez2015}) and Fu (\cite[Theorem 3.1, 4.8, 4.13]{Fu2017}) they have been showed that, for certain types of algebras, the set of positive $c$-vectors correspond to the dimension vectors of exceptional objects, that is, bricks without non-trivial self-extensions.
However, in corollary \ref{cor:c-dim} we only are able to prove that $c$-vectors correspond to dimension vectors of bricks. 
The following example shows that we can not do better.

\begin{example}
Let $A$ be the path algebra of the quiver 
\[
\begin{tikzcd}
    1\ar["\alpha"]{r} & 2\ar[loop above, "\beta"]
\end{tikzcd}
\]
modulo the ideal generated by all the paths of length two.
Then the indecomposable projective modules in $\mod A$ are $P(1)=\rep{1\\2}$ and $P(2)=\rep{2\\2}$.
Therefore, the pair $\left(\rep{1\\2}\oplus \rep{2\\2}, 0\right)$ is trivially a $\tau$-tilting pair.
Moreover, is easy to see that 
$$G_{\left(\rep{1\\2}\oplus \rep{2\\2}, 0\right)}=C_{\left(\rep{1\\2}\oplus \rep{2\\2}, 0\right)}=\left(\begin{matrix}
1&0 \\ 0&1
\end{matrix}\right).$$
Now, we have two positive $c$-vectors $(1,0)$ and $(0,1)$, corresponding to the dimension vectors to the simple modules $S(1)=\rep{1}$ and $S(2)=\rep{2}$, respectively.
Then the simple $A$-module $S(2)$ is certainly a brick, but not an exceptional object in $\mod A$ since 
the projective module $P(2)$ is a self-extension of $S(2)$.
\end{example}

Hence the problem of classifying all the algebras whose positive $c$-vectors correspond to its exceptional objects arises naturally.

\subsection{$c$-vectors and functorially finite torsion pairs}

Let $(M,P)$ be a $\tau$-tilting pair. 
In this subsection we study the relation between the set of bricks $\mathcal{B}_{(M,P)}$ associated to $(M,P)$ and the torsion pair $(\Fac M, M^{\perp})$ induced by it. 

In order to state one of the main theorems of this paper, we need to introduce some terminology and notation.

Given two $A$-modules $X$ and $Y$, we say that they are \textit{$\Hom$-orthogonal} if 
$$\Hom_A(X,Y)=\Hom_A(Y,X)=0.$$
The following proposition is a direct consequence of the definition of theorem \ref{th:c-matrix}.

\begin{lemma}
Let $(M,P)$ be a $\tau$-tilting pair, $(\Fac M, M^\perp)$ be the torsion pair induced by it and $B_i \in \B_{(M,P)}$.
If the dimension vector $[B_i]$ of $B_i$ is a column of $C_{(M,P)}$ then $B_i\in \Fac M$.
Dually, if the opposite of the dimension vector $[B_i]$ of $B_i$ is a column of $C_{(M,P)}$ then $B_i\in M^\perp$.

Moreover, if $B_s, B_t$ are two different bricks in $\B_{(M,P)}$ such that their dimension vectors are positive columns of $C_{(M,P)}$, then $B_i$ and $B_j$ are $\Hom$-orthogonal.
\end{lemma}

\begin{proof}
The first part of the statement is a particular case proposition \ref{prop:SinFac} and its dual. 

Now we prove the moreover part of the statement. 
By construction $B_s$ is a $\theta_s$-stable module.
Then \cite[Proposition 3.13]{BSTwc} implies that 
$$B_s\in (\bigoplus_{i\neq s} M_i)^\perp \cap {^\perp}(\bigoplus_{i\neq s} \tau M_i) \cap (\bigoplus_{j\neq s} P_j)^\perp. $$
Moreover, the fact that $B_s \in \Fac M$ yields an epimorphism $p_s: M_s \to B_s$. 
The same argument shows the existence of an epimorphism $p_t: M_t \to B_t$.

Now, every morphism $f\in \Hom_A(B_t, B_s)$ can be composed to the left with $p_t$ to get a morphism $p_tf:M_t \to B_s$.
But $M_t$ is a direct summand of $\oplus_{i\neq s}M_i$, so $\Hom_A(M_t, B_s)=0$.
Hence $f=0$ because $p_t$ is an epimorphism. 
The fact that $\Hom_A(B_s, B_t)=0$ is shown in the same fashion.
This finishes the proof. 
\end{proof}

Consider a set $\mathcal{N}=\{N_1, \dots, N_t\}$ of $A$-modules and let $N=\bigoplus_{i=1}^t N_i$.
Then we denote by $T(\mathcal{N})$ the full subcategory of $\mod A$ having as objects that can be filtered by objects in $\Fac N$. 
More precisely, 
$$T(\mathcal{N}):=\{X\in \mod A: 0=X_0\subset X_1 ...\subset X_s=X \text{ where $X_i/X_{i-1} \in \Fac N$}\}.$$
Is easy to see that $T(\mathcal{N})$ is closed under quotients and extensions, which implies that $T(\mathcal{N})$ is always a torsion pair.
Moreover, as shown in \cite[Proposition 3.3]{DIJ}, $T(\mathcal{N})$ is the minimal torsion class containing $\mathcal{N}$.

Let $N$ be an $A$-module and $N'$ be a self-extension of $N$.
Then is easy to see that $T(N)=T(N')$. 
Therefore some redundancies needs to be avoided.
In order to do that, from now on we restrict ourselves to the case where every $B\in\mathcal{N}$ is a brick.

Then, a natural question to ask is the following:
Given a torsion class $\T$ in $\mod A$, what is minimal collection of bricks $\mathcal{N}$ such that $\T=T(\mathcal{N})$?

This was answered by Barnard, Carrol and Zhu in \cite{BCZ} as follows.

\begin{theorem}\cite[Theorem 1.0.8]{BCZ}
Let $\T$ be a torsion class in $\emph{mod} A$.
Then the minimal set $\mathcal{N}_{\T}\subset\T$ such that $T(\mathcal{N})=\T$ is the set which is maximal for the following properties:
\begin{itemize}
\item $\mathcal{N}\subset \T$;
\item $\Hom_A(B,B')=0$ for all $B, B' \in \mathcal{N}$.
\end{itemize}
\end{theorem}

\begin{remark}\label{rmk:unicity}
Note that the maximality of $\mathcal{N}_{\T}$ implies that it is unique for every torsion class $\T$.
\end{remark}

Now we are able to state and prove the main theorem of this paper, which gives a characterisation of all modules whose dimension vectors are positive $c$-vectors.

\begin{theorem}\label{th:FacM}
Let $(M,P)$ be a $\tau$-tilting pair and define $\B_{(M,P)}^+\subset\B_{(M,P)}$ as 
$$\B_{(M,P)}^+:=\{B_i\in\B_{(M,P)}: [B_i] \text{ is a column of $C_{(M,P)}$}\}.$$
Then $T(\B_{(M,P)}^+)=\emph{Fac} M$.

Moreover, the reciprocal also holds. 
Namely, if $\mathcal{N}$ is such that:
\begin{itemize}
\item $B$ is a brick for every $B\in\mathcal{N}$;
\item $\Hom_A(B,B')=0$ for all $B, B' \in \mathcal{N}$;
\item $T(\mathcal{N})$ is functorially finite. 
\end{itemize}
Then $\mathcal{N}=\B_{(M,P)}^+$ for some $\tau$-tilting pair $(M,P)$.
\end{theorem}

\begin{proof}
First, let $(M,P)$ be a $\tau$-tilting pair and let $\B_{(M,P)}^+$ as in the statement.
Then, it follows from theorem \ref{stablemodcat} that $N_s$ can be filtered by $B_s$ for all $B_s\in\B_{(M,P)}^+$, where $N_i$ minimal generator of the category of $\theta_s$-semistable objects.
This implies that $$T(\B_{(M,P)}^+)=T(\{N_s: N_s \text{ is the generator of the $\theta_s$-semistable modules}\})$$
Moreover, \cite[Lemma 2.3]{BSTwc} implies that $N_i$ is the cokernel of right $add(\oplus_{i\neq s}M_i)$-approximation of $M_s$. 
Therefore $T(\B_{(M,P)^+})=\Fac M$ by \cite[Lemma 3.7]{DIJ}.

Now we show the moreover part of the statement.
Let $\mathcal{N}$ as in the statement. 
Is clear that $\mathcal{N}$ is minimal set inducing $T(\mathcal{N})$.
Moreover, since $T(\mathcal{N})$ is functorially finite, then $T(\mathcal{N})=\Fac M$ for some $\tau$-tilting pair $(M,P)$ by \cite[Theorem 2.7]{AIR}.
On the other hand, we just showed that $\Fac M=T(\B_{(M,P)}^+)$.
Hence $\mathcal{N}=\B_{(M,P)}^+$ by remark \ref{rmk:unicity}.
This finishes the proof.
\end{proof}

\begin{remark}
Note that a single positive $c$-vector can be the dimension vector of multiple bricks. 
For instance, let $A$ be the path algebra of the quiver 
$$\begin{tikzcd}
1 \ar[r, "\alpha", bend left =20] & 2 \ar[l, "\beta", bend left =20] \\
\end{tikzcd}$$ 
bounded by its radical square. 
Then $(M_{1},P_{1})=(\rep{1\\2}\oplus{1}, 0)$ and $(M_{2}, P_{2})=(\rep{2\\1}\oplus\rep{2}, 0)$ are two $\tau$-tilting pairs in $\mod A$ whose $C$-matrices are 
$$C_{(M_{1},P_{1})}=\begin{pmatrix} 1 & 0 \\ 1 & -1 \end{pmatrix} \qquad \text{ and } \qquad C_{(M_{2},P_{2})}=\begin{pmatrix} 1 & -1 \\ 1 & 0 \end{pmatrix}$$
respectively. 
Therefore $(M_{1}, P_{1})$ and $(M_{2}, P_{2})$ induce the same positive $c$-vector $(1,1)$. 

On the other hand, the brick $B_{1}$ induced by $(M_{1},P_{1})$ is $B_{1}=\rep{1\\2}$ while the brick $B_{2}$ induced by $(M_{2},P_{2})$ is $B_{2}=\rep{2\\1}$.
\end{remark}

\section{The wall and chamber structure of an algebra}\label{sec:wallandchamber}

If one looks at the Definition \ref{stability} of King's stability condition, one see that there are two directions to study them: 
either one fix a functional $\theta$ and study the category of $\theta$-semistable modules, as is done in Theorem \ref{stablemodcat}, or one fix a module $M$ and study the functionals $\theta$ making $M$ a $\theta$-semistable module. 
In this subsection we continue the work of \cite{BSTwc} considering the second option. 

\begin{definition}
The \textit{stability space} of an $A$-module $M$ is $$\mathfrak{D}(M)=\{\theta\in\mathbb{R}^n : M \text{ is $\theta$-semistable}\}.$$
Moreover the stability space $\mathfrak{D}(M)$ of $M$ is said to be a \textit{wall} when $\mathfrak{D}(M)$ has codimension one. 
In this case we say that $\mathfrak{D}(M)$ is the wall defined by $M$.
\end{definition}

Note that not every $\theta$ belongs to the stability space $\mathfrak{D}(M)$ for some nonzero module $M$. For instance, is easy to see that $\theta=(1,1, \dots, 1)$
is an example of such a functional for every algebra $A$. 
This leads to the following definition.

\begin{definition}
Let $A$ be an algebra such that $rk(K_0(A))=n$ and
$$\mathfrak{R}=\mathbb{R}^n\setminus\overline{\bigcup\limits_{\substack{ M\in \mod A}}\mathfrak{D}(M)}$$
be the maximal open set of all $\theta$ having no $\theta$-semistable modules other that the zero object. 
Then a dimension $n$ connected component $\mathfrak{C}$ of $\mathfrak{R}$ is called a \textit{chamber} and this partition of $\mathbb{R}^{n}$ is known as the \textit{wall and chamber structure} of $A$.
\end{definition}

One of the main objectives of \cite{BSTwc} was to describe the wall and chamber structure of an algebra using $\tau$-tilting theory. 
This was partially achieved, when they determined that every $\tau$-tilting pairs induce a chamber and they described the walls surrounding that chamber. See \cite[Corollary 3.18]{BSTwc}

In this context, one can use the results in this paper to go one step further, as shown in the following result.

\begin{corollary}\label{cor:walls}
Let $(M,P)$ be a $\tau$-tilting pair and $\mathfrak{C}_{(M,P)}$ be the chamber induced by it. 
Then the walls surrounding $\mathfrak{C}_{(M,P)}$ are defined by the $c$-vectors corresponding to $(M,P)$, that is, the columns of the matrix $C_{(M,P)}$.
\end{corollary}

\begin{proof}
It was shown in \cite[Corollary 3.18]{BSTwc} that walls surrounding $\mathfrak{C}_{(M,P)}$ are defined by the dimension vectors of $\theta_r$-(semi)stable modules, where $1 \leq r \leq n$.
Then the result follows directly from theorem \ref{th:c-matrix}.
\end{proof}

It was already remarked in \cite[Section 4]{BSTwc} that the wall and chamber structure is dual to the exchange graph of $\tau$-tilting pairs and that this duality induces a brick labelling of the exchange graph.
Therefore, corollary \ref{cor:walls} implies that we can label the arrows in the exchange lattice of $\tau$-tilting pairs by $c$-vectors.
Note that this labelling is coincides with the labelling studied in \cite{Demonet2017}.

\begin{corollary}\label{cor:labeling}
Let $A$ be an algebra. Then every arrow in the exchange graph of $\tau$-tilting pairs can be labeled with a positive $c$-vector.
\end{corollary}

\begin{proof}
Suppose that $(M,P)$ and $(M',P')$ are two $\tau$-tilting pairs such that one is a mutation of the other. 
Without loss of generality we can suppose that $\Fac M' \subset \Fac M$
Then \cite[Proposition 3.17]{BSTwc} implies that the chambers $\Ch_{(M,P)}$ and $\Ch_{(M',P')}$ inducing by them share a wall $\mathfrak{D}(N)$ for some module $N$.
But, corollary \ref{cor:walls} implies that $N$ can be taken to be a brick, which is unique up to isomorphism.
Therefore the dimension vector of $N$ is a positive $c$-vector for the $\tau$-tilting pair $(M,P)$ by proposition \ref{prop:SinFac} and theorem \ref{th:c-matrix}.
This finishes the proof.
\end{proof}

\section{A nice example}\label{sc:Example}

\begin{table}
 \begin{center}
  \scalebox{0.7}{
	\begin{tabular}{ | c | c | c | c | c | c |}
		\hline
		
		$(M,P)$	&	$G_{(M,P)}$	& 	$C_{(M,P)}=(G^{T}_{(M,P)})^{-1}$	& 	Positive $c$-vectors	&	$\B^+_{(M,P)}$ & $\Fac M = T(\B^+_{(M,P)})$	\\ 
		\hline
		$\left(\rep{1\\2} \oplus \rep{2\\3} \oplus \rep{3} , 0 \right)$	& 
		$ \begin{pmatrix}1&0&0 \\ 0&1&0 \\ 0&0&1\end{pmatrix} $	& 
		$\left( \begin{matrix}1&0&0 \\ 0&1&0 \\ 0&0&1\end{matrix} \right)$	& 
		$\left\{ \left( \begin{matrix} 1\\0\\0 \end{matrix} \right), \left( \begin{matrix} 0\\1\\0 \end{matrix}\right), \left( \begin{matrix} 0\\0\\1\end{matrix} \right) \right\}$	&
		$\{\rep{1}$, $\rep{2}$, $\rep{3}\}$ & $\mod A$\\
		\hline

		$\left( \rep{1\\2} \oplus \rep{2\\3} \oplus \rep{2} , 0 \right)$	& 
		$ \begin{pmatrix}1&0&0 \\ 0&1&1 \\ 0&0&-1\end{pmatrix} $	&
		$ \begin{pmatrix}1&0&0 \\ 0&1&0 \\ 0&1&-1\end{pmatrix} $	&
		$ \left\{ \begin{pmatrix} 1\\0\\0 \end{pmatrix}, \begin{pmatrix} 0\\1\\1 \end{pmatrix} \right\}$	&
		$\left\{ \rep{1}, \rep{2\\3} \right\}$ & $add\left\{ \rep{1\\2} \oplus \rep{2\\3} \oplus \rep{2} \oplus \rep{1} \right\}$	\\
		\hline
		
		$\left(\rep{1\\2} \oplus \rep{1} \oplus \rep{3} , 0 \right)$	&
		$\left( \begin{matrix}1&1&0 \\ 0&-1&0 \\ 0&0&1\end{matrix} \right)$	&
		$ \begin{pmatrix}1&0&0 \\ 1&-1&0 \\ 0&0&1\end{pmatrix} $	&
		$ \left\{ \begin{pmatrix} 1\\1\\0 \end{pmatrix}, \begin{pmatrix} 0\\0\\1 \end{pmatrix} \right\}$	&
		$\left\{ \rep{1\\2}, \rep{3} \right\}$	& $add\left\{ \rep{1\\2} \oplus \rep{1} \oplus \rep{3} \right\}$ \\
		\hline
		
		$\left( \rep{2\\3} \oplus \rep{3}, \rep{1\\2} \right)$			& 
		$\left( \begin{matrix}0&0&-1 \\ 1&0&0 \\ 0&1&0\end{matrix} \right)$	&
		$ \begin{pmatrix}0&0&-1 \\ 1&0&0 \\ 0&1&0\end{pmatrix} $	&
		$ \left\{ \begin{pmatrix} 0\\1\\0 \end{pmatrix}, \begin{pmatrix} 0\\0\\1 \end{pmatrix} \right\}$	&
		$\left\{ \rep{2}, \rep{3} \right\}$ & $add\left\{ \rep{2\\3} \oplus \rep{3} \oplus \rep{2}\right\}$	\\
		\hline
	
		$\left( \rep{1\\2} \oplus \rep{2}, \rep{3} \right)$			& 
		$\left( \begin{matrix}1&0&0 \\ 0&1&0 \\ 0&-1&-1\end{matrix} \right)$	&
		$ \begin{pmatrix}1&0&0 \\ 0&1&-1 \\ 0&0&-1\end{pmatrix} $	&
		$ \left\{ \begin{pmatrix} 1\\0\\0 \end{pmatrix}, \begin{pmatrix} 0\\1\\0 \end{pmatrix} \right\}$	&
		$\left\{ \rep{1}, \rep{2} \right\}$	& $add\left\{ \rep{1\\2} \oplus \rep{2} \oplus \rep{1} \right\}$\\
		\hline
		
		$\left( \rep{1} \oplus \rep{3}, \rep{2\\3} \right)$			& 
		$ \begin{pmatrix}1&0&0 \\ -1&0&-1 \\ 0&1&0\end{pmatrix} $	&
		$ \begin{pmatrix}1&0&-1 \\ 0&0&-1 \\ 0&1&0\end{pmatrix} $	&
		$ \left\{ \begin{pmatrix} 1\\0\\0 \end{pmatrix}, \begin{pmatrix} 0\\0\\1 \end{pmatrix} \right\}$	&
		$\left\{ \rep{1}, \rep{3} \right\}$	& $add\left\{ \rep{1} \oplus \rep{3}\right\}$ \\
		\hline
		
		$\left( \rep{3}, \rep{1\\2} \oplus \rep{2\\3} \right)$			& 
		$\left( \begin{matrix}0&-1&0 \\ 0&0&-1 \\ 1&0&0\end{matrix} \right)$	&
		$ \begin{pmatrix}0&-1&0 \\ 0&0&-1 \\ 1&0&0\end{pmatrix} $	&
		$ \left\{ \begin{pmatrix} 0\\0\\1 \end{pmatrix} \right\}$	&
		$\left\{ \rep{3} \right\}$	&  $add\left\{ \rep{3} \right\}$\\
		\hline
		
		$\left( \rep{1\\2} \oplus \rep{1}, \rep{3} \right)$			& 
		$\left( \begin{matrix}1&1&0 \\ 0&-1&0 \\ 0&0&-1\end{matrix} \right)$	&
		$ \begin{pmatrix}1&0&0 \\ 1&-1&0 \\ 0&0&-1\end{pmatrix} $	&
		$ \left\{ \begin{pmatrix} 1\\1\\0 \end{pmatrix} \right\}$	&
		$\left\{ \rep{1\\2} \right\}$ & $add\left\{ \rep{1\\2} \oplus \rep{1} \right\}$	\\
		\hline
		
		$\left( \rep{2\\3} \oplus \rep{2}, \rep{1\\2} \right)$			& 
		$\left( \begin{matrix}0&0&-1 \\ 1&1&0 \\ 0&-1&0\end{matrix} \right)$	&
		$ \begin{pmatrix} 0&0&-1 \\ 1&0&0 \\ 1&-1&0\end{pmatrix} $	&
		$ \left\{  \begin{pmatrix} 0\\1\\1 \end{pmatrix} \right\}$	&
		$\left\{ \rep{2\\3} \right\}$	& $add\left\{ \rep{2\\3} \oplus \rep{2} \right\}$\\
		\hline
		
		$\left( \rep{2}, \rep{1\\2} \oplus \rep{3} \right)$			&
		$\left( \begin{matrix}0&-1&0 \\ 1&0&0 \\ -1&0&-1\end{matrix} \right)$	&
		$ \begin{pmatrix}0&-1&0 \\ 1&0&-1 \\ 0&0&-1\end{pmatrix} $	&
		$ \left\{ \begin{pmatrix} 0\\1\\0 \end{pmatrix} \right\}$	&
		$\left\{ \rep{2} \right\}$ & $add\left\{ \rep{2} \right\}$	\\
		\hline
		
		$\left( \rep{1}, \rep{2\\3} \oplus \rep{3} \right)$			& 
		$\left( \begin{matrix}1&0&0 \\ -1&-1&0 \\ 0&0&-1\end{matrix} \right)$	&
		$ \begin{pmatrix}1&-1&0 \\ 0&-1&0 \\ 0&0&-1\end{pmatrix} $	&
		$ \left\{ \begin{pmatrix} 1\\0\\0 \end{pmatrix} \right\}$	&
		$\left\{ \rep{1} \right\}$	& $add\left\{ \rep{1}\right\}$ \\
		\hline
		
		$\left( 0 , \rep{1\\2} \oplus \rep{2\\3} \oplus \rep{3} \right)$	& 
		$\left( \begin{matrix}-1&0&0 \\ 0&-1&0 \\ 0&0&-1\end{matrix} \right)$	&
		$\left( \begin{matrix}-1&0&0 \\ 0&-1&0 \\ 0&0&-1\end{matrix} \right)$	&
		$\varnothing$	&
		$\varnothing$ & $add\left\{ 0 \right\}$	\\
		\hline

	\end{tabular}}
 \end{center}
 \caption{$\tau$-tilting pairs with their corresponding positive $c$-vectors, bricks and torsion classes}
 \label{haytabla}
\end{table}

We finish illustrating theorem \ref{th:c-matrix}, theorem \ref{th:FacM}, corollary \ref{cor:walls} and corollary \ref{cor:labeling} in the case of one particular algebra.

\begin{example}\label{A3tilted}
We consider the case of $A$, the path algebra of the quiver $$\xymatrix{1\ar^{\alpha}[r] & 2\ar^{\beta}[r] &3}$$ modulo the ideal generated by the relation $\alpha\beta$. 

This algebra has 12 different $\tau$-tilting pairs. 
We list them in Table \ref{haytabla} along with their corresponding $G$-matrices, $C$-matrices, positive $c$-vectors, $\mathfrak{B}_{(M,P)}^+$ and torsion classes, as proved in theorems \ref{th:c-matrix} and theorem \ref{th:FacM}.

Also, one can see in Figure \ref{labeled} the exchange graph of $\tau$-tilting pairs in which every arrow is labeled with a positive $c$-vector, as showed in corollary \ref{cor:labeling}. 

Finally, in figure \ref{w&c} one can see a representation of the wall and chamber structure of $A$ in the style of \cite[Example 4.0.2]{Igusa2009}. 
In this figure, each circle or arc correspond to the stereographic projection intersection of each wall with the unit sphere in $\mathbb{R}^{3}$. 
Note that in this projection the point at infinity correspond to the point $(1,1,1)$ in $\mathbb{R}^{3}$. Moreover, the arcs and circles are labeled by the $c$-vectors which are perpendicular to them and each chamber is tagged by the $\tau$-tilting pair that generates it.

Note that, given a $\tau$-tilting pair $(M,P)$, the positive columns of the $C$-matrix $C_{(M,P)}$ coincide with the non-convex arc surrounding $\mathfrak{C}_{(M,P)}$, while negative $c$-vectors correspond to the convex ones. 

Also, as we pointed out already, the exchange graph of $\tau$-tilting pairs can be embedded into the wall and chamber structure of the algebra, where a mutation corresponds to crossing a wall. 
To illustrate this phenomenon, we have coloured the arrows in figure \ref{labeled} with the colours of the walls that we are crossing at each mutation. 
Remark that arrows go from the chamber in which the given $c$-vector is positive to a chamber where it is negative. 

\begin{figure}
 \begin{center}
  \scalebox{1}{\begin{tikzcd}[ ampersand replacement=\&, row sep = huge]
															\& \left(\rep{1\\2} \oplus \rep{2\\3} \oplus \rep{3} , 0 \right)\ar[dr, "\tiny{\begin{pmatrix}1\\0\\0\end{pmatrix}}"  description, color=red]\ar[d, color=black!60!green, "\tiny{\begin{pmatrix}0\\1\\0\end{pmatrix}}" description]\ar[dl, "\tiny{\begin{pmatrix}0\\0\\1\end{pmatrix}}" ' description, color=blue]	\&										\\
				\left( \rep{1\\2} \oplus \rep{2\\3} \oplus \rep{2} , 0 \right)\ar[dr, "\tiny{\begin{pmatrix}1\\0\\0\end{pmatrix}}"description, color=red]\ar[d, "\tiny{\begin{pmatrix}0\\1\\1\end{pmatrix}}"' description, color=violet]	\& \left(\rep{1\\2} \oplus \rep{1} \oplus \rep{3} , 0 \right)\ar[ddl, gray, "\tiny{\begin{pmatrix}0\\0\\1\end{pmatrix}}"' description, color=blue]\ar[dd, gray, shift left=1, "\tiny{\begin{pmatrix}1\\1\\0\end{pmatrix}}" near start, color=brown]	\& \left( \rep{2\\3} \oplus \rep{3}, \rep{1\\2} \right)\ar[dl, "\tiny{\begin{pmatrix}0\\0\\1\end{pmatrix}}" ' description, color=blue]\ar[ddd, "\tiny{\begin{pmatrix}0\\1\\0\end{pmatrix}}" description, color=black!60!green]	\\
				\left( \rep{1\\2} \oplus \rep{2}, \rep{3} \right)\ar[d, "\tiny{\begin{pmatrix}0\\1\\0\end{pmatrix}}"' description, color=black!60!green]\ar[ddr, "\tiny{\begin{pmatrix}1\\0\\0\end{pmatrix}}" description, color=red]			\& \left( \rep{2\\3} \oplus \rep{2}, \rep{1\\2} \right)\ar[dd, shift left=1, bend left=50, "\tiny{\begin{pmatrix}0\\1\\1\end{pmatrix}}" description, crossing over, color=violet]		\&										\\
				\left( \rep{1\\2} \oplus \rep{1}, \rep{3} \right)\ar[d, "\tiny{\begin{pmatrix}1\\1\\0\end{pmatrix}}" ' description, color=brown]			\& \left( \rep{1} \oplus \rep{3}, \rep{2\\3} \right)\ar[dl, gray, "\tiny{\begin{pmatrix}0\\0\\1\end{pmatrix}}" ' description, color=blue] \ar[dr, "\tiny{\begin{pmatrix}1\\0\\0\end{pmatrix}}" description, color=red]			\&										\\
				\left( \rep{1}, \rep{2\\3} \oplus \rep{3} \right)\ar[dr, "\tiny{\begin{pmatrix}1\\0\\0\end{pmatrix}}" description, color=red]			\& \left( \rep{2}, \rep{1\\2} \oplus \rep{3} \right)\ar[d, "\tiny{\begin{pmatrix}0\\1\\0\end{pmatrix}}" description, color=black!60!green]			\& \left( \rep{3}, \rep{1\\2} \oplus \rep{2\\3} \right)\ar[dl, "\tiny{\begin{pmatrix}0\\0\\1\end{pmatrix}}"' description, color=blue]	\\
					\& \left( 0 , \rep{1\\2} \oplus \rep{2\\3} \oplus \rep{3} \right)	\&	\\																									
				\end{tikzcd}}
 \end{center}
 \caption{$\tau$-tilting exchange graph of $A$ with bricks labelling}
 \label{labeled}
\end{figure}
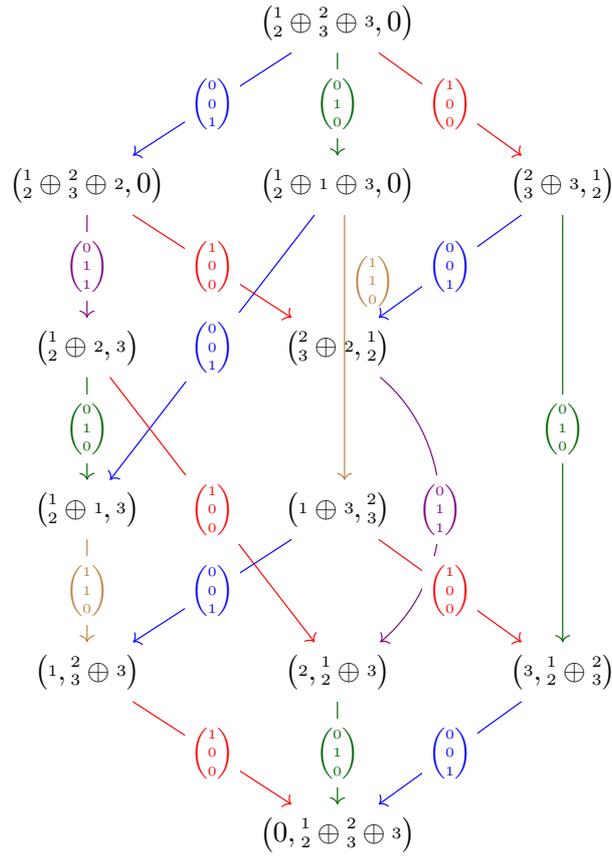

\begin{figure}
\begin{center}
\scalebox{0.9}{
\begin{tikzpicture}[line cap=round,line join=round,>=triangle 45,x=1.0cm,y=1.0cm]
\clip(-7.08,-7.22) rectangle (9.76,7.7);
\draw [line width=0.6pt, color=red] (0.,3.) circle (4.cm);
\draw [line width=0.6pt, color=black!60!green] (-2.598076211353316,-1.5) circle (4.cm);
\draw [line width=0.6pt, color=blue] (2.598076211353316,-1.5) circle (4.cm);
\draw [shift={(0.5836956521739132,-1.5)}, line width=0.6pt, color=violet] plot[domain=-1.76040874305534:1.760408743055339,variable=\t]({1.*3.0968856314637665*cos(\t r)+0.*3.0968856314637665*sin(\t r)},{0.*3.0968856314637665*cos(\t r)+1.*3.0968856314637665*sin(\t r)});
\draw [shift={(-1.0584754935143137,1.1666666666666674)},line width=0.6pt, color=brown]  plot[domain=2.7748867490119555:5.602693660560826,variable=\t]({1.*3.0792014356780046*cos(\t r)+0.*3.0792014356780046*sin(\t r)},{0.*3.0792014356780046*cos(\t r)+1.*3.0792014356780046*sin(\t r)});
\begin{scriptsize}
\draw[color=red] (-4,5) node {$\begin{pmatrix}1\\0\\0\end{pmatrix}^{\perp}$};
\draw[color=black!60!green] (-5.9,-4.5) node {$\begin{pmatrix}0\\1\\0\end{pmatrix}^{\perp}$};
\draw[color=blue] (5.9,-4.5) node {$\begin{pmatrix}0\\0\\1\end{pmatrix}^{\perp}$};
\draw[color=violet] (3.5,-4) node {$\begin{pmatrix}0\\1\\1\end{pmatrix}^{\perp}$};
\draw[color=brown] (-4,-1) node {$\begin{pmatrix}1\\1\\0\end{pmatrix}^{\perp}$};
\draw (0,0) node {$\mathfrak{C}_{\left(0, \rep{1\\2}\oplus\rep{2\\3}\oplus\rep{3}\right)}$};
\draw (6,4) node {$\mathfrak{C}_{\left(\rep{1\\2}\oplus\rep{2\\3}\oplus\rep{3} , 0\right)}$};
\draw (0,4) node {$\mathfrak{C}_{\left(\rep{2\\3}\oplus\rep{3}, \rep{1\\2}\right)}$};
\draw (-1.5,1) node {$\mathfrak{C}_{\left(\rep{3}, \rep{1\\2}\oplus\rep{2\\3}\right)}$};
\draw (2.7,1.7) node {$\mathfrak{C}_{\left(\rep{2\\3}\oplus\rep{2}, \rep{1\\2}\right)}$};
\draw (2,0.3) node {$\mathfrak{C}_{\left(\rep{2}, \rep{1\\2}\oplus\rep{3}\right)}$};
\draw (-0.5,-1.4) node {$\mathfrak{C}_{\left(\rep{1}, \rep{2\\3}\oplus\rep{3}\right)}$};
\draw (-2.5,-0.9) node {$\mathfrak{C}_{\left(\rep{1}\oplus\rep{3},\rep{2\\3}\right)}$};
\draw (5,-1) node {$\mathfrak{C}_{\left( \rep{1\\2}\oplus\rep{2\\3}\oplus\rep{2},0\right)}$};
\draw (-4,-3) node {$\mathfrak{C}_{\left( \rep{1\\2}\oplus\rep{3}\oplus\rep{1},0\right)}$};
\draw (0,-3) node {$\mathfrak{C}_{\left( \rep{1\\2}\oplus\rep{1},\rep{3}\right)}$};
\draw (2.5,-1.5) node {$\mathfrak{C}_{\left( \rep{1\\2}\oplus\rep{2},\rep{3}\right)}$};
\end{scriptsize}
\end{tikzpicture}}

\caption{The wall and chamber structure of $A$ labeled by $c$-vectors}
\label{w&c}
\end{center}
\end{figure}
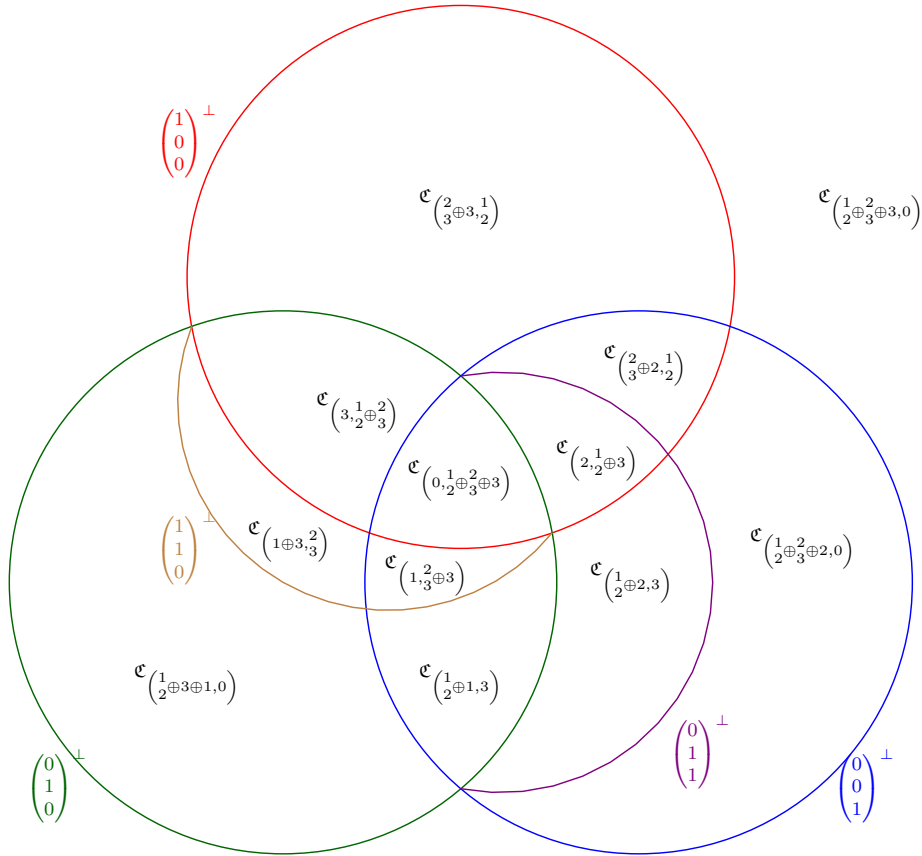
\end{example}

\newpage
\bibliography{Mendeley.bib}{}

\begin{thebibliography}{10}

\bibitem{AIR}
T.~Adachi, O.~Iyama, and I.~Reiten.
\newblock {{$\tau$}-tilting theory}.
\newblock {\em Compositio Mathematica}, 150(3):415--452, 2014.

\bibitem{Asai2016}
S.~Asai.
\newblock {Semibricks}.
\newblock {\em arXiv}, 10(http://arxiv.org/abs/1610.05860), 2016.

\bibitem{Auslander1985}
M.~Auslander and I.~Reiten.
\newblock {Modules determined by their composition factors}.
\newblock {\em Illinois Journal of Mathematics}, 29(2):280--301, 1985.

\bibitem{Auslander1981}
M.~Auslander and S.~O. Smal{\o}.
\newblock {Almost split sequences in subcategories}.
\newblock {\em Journal of Algebra}, 69(2):426--454, 8 1981.

\bibitem{BCZ}
E.~Barnard, A.~T. Carroll, and S.~Zhu.
\newblock {Minimal inclusions of torsion classes}.
\newblock {\em arXiv}, (http://arxiv.org/abs/1710.08837):1--25, 2017.

\bibitem{Bridgeland2016a}
T.~Bridgeland.
\newblock {Scattering diagrams, Hall algebras and stability conditions}.
\newblock {\em Algebr. Geom.}, 4(November 2016):523--561, 2016.

\bibitem{BSTwc}
T.~Br{\"{u}}stle, D.~Smith, and H.~Treffinger.
\newblock {Wall and Chamber Structure for finite-dimensional Algebras}.
\newblock {\em arXiv}, (http://arxiv.org/abs/1805.01880):1--24, 2018.

\bibitem{Buan2006a}
A.~B. Buan, R.~Marsh, M.~Reineke, I.~Reiten, and G.~Todorov.
\newblock {Tilting theory and cluster combinatorics}.
\newblock {\em Advances in Mathematics}, 204(2):572--618, 2006.

\bibitem{Caldero2005}
P.~Caldero, F.~Chapoton, and R.~Schiffler.
\newblock {Quivers with relations arising from clusters (An case)}.
\newblock {\em Transactions of the American Mathematical Society},
  358(03):1347--1364, 2005.

\bibitem{Chavez2015}
A.~N. Ch{\'{a}}vez.
\newblock {On the c-Vectors of an Acyclic Cluster Algebra}.
\newblock {\em International Mathematics Research Notices}, 2015(6):1590--1600,
  1 2015.

\bibitem{DIJ}
L.~Demonet, O.~Iyama, and G.~Jasso.
\newblock {{$\tau$} -Tilting Finite Algebras, Bricks, and g-Vectors}.
\newblock {\em International Mathematics Research Notices}, 2017(00):1--41,
  2015.

\bibitem{Demonet2017}
L.~Demonet, O.~Iyama, N.~Reading, I.~Reiten, and H.~Thomas.
\newblock {Lattice theory of torsion classes}.
\newblock {\em arXiv}, (http://arxiv.org/abs/1711.01785):1--65, 2017.

\bibitem{Deny2008}
R.~Deny and B.~Keller.
\newblock {On the combinatorics of rigid objects in 2-calabi-yau categories}.
\newblock {\em International Mathematics Research Notices}, 2008(1), 2008.

\bibitem{Derksen2010}
H.~Derksen, J.~Weyman, and A.~Zelevinsky.
\newblock {Quivers with potentials and their representations II: Applications
  to cluster algebras}.
\newblock {\em Journal of the American Mathematical Society}, 23(3):749--749,
  2010.

\bibitem{Fomin2001}
S.~Fomin and A.~Zelevinsky.
\newblock {Cluster algebras I: Foundations}.
\newblock {\em J. Amer. Math. Soc.}, 15(2):497--529, 2001.

\bibitem{Fomin2003}
S.~Fomin and a.~Zelevinsky.
\newblock {Cluster algebras II: Finite type classification}.
\newblock {\em Inventiones Mathematicae}, 0070685:1--50, 2003.

\bibitem{Fomin2007}
S.~Fomin and A.~Zelevinsky.
\newblock {Cluster algebras IV: Coefficients}.
\newblock {\em Compositio Mathematica}, 143(1):112--164, 2 2006.

\bibitem{Fu2017}
C.~Fu.
\newblock {c-vectors via {$\tau$}-tilting theory}.
\newblock {\em Journal of Algebra}, 473:194--220, 2017.

\bibitem{Gross2017}
M.~Gross, P.~Hacking, S.~Keel, and M.~Kontsevich.
\newblock {Canonical bases for cluster algebras}.
\newblock {\em Journal of the American Mathematical Society}, 31(2):497--608,
  11 2017.

\bibitem{Igusa2009}
K.~Igusa, K.~Orr, G.~Todorov, and J.~Weyman.
\newblock {Cluster complexes via semi-invariants}.
\newblock {\em Compositio Mathematica}, 145(4):1001--1034, 2009.

\bibitem{Jasso2015}
G.~Jasso.
\newblock {Reduction of {$\tau$}-tilting modules and torsion pairs}.
\newblock {\em International Mathematics Research Notices},
  2015(16):7190--7237, 2015.

\bibitem{JY}
P.~Jorgensen and M.~Yakimov.
\newblock {$c$-vectors of 2-Calabi--Yau categories and Borel subalgebras of $\mathfrak{sl}_{\infty}$}.
\newblock pages 1--38, {\em arXiv}, (https://arxiv.org/abs/1710.06492), 2017.


\bibitem{KING1994}
A.~D. King.
\newblock {Moduli of Representations of Finite Dimensional Algebras}.
\newblock {\em The Quarterly Journal of Mathematics}, 45(4):515--530, 1994.

\bibitem{Musiker2013}
G.~Musiker, R.~Schiffler, and L.~Williams.
\newblock {Bases for cluster algebras from surfaces}.
\newblock {\em Compositio Mathematica}, 149(2):217--263, 2013.

\bibitem{Nagao2013}
K.~Nagao.
\newblock {Donaldson-Thomas theory and cluster algebras}.
\newblock {\em Duke Mathematical Journal}, 162(7):1313--1367, 5 2013.

\bibitem{NajeraChavez2013}
A.~N{\'{a}}jera~Ch{\'{a}}vez.
\newblock {c-vectors and dimension vectors for cluster-finite quivers}.
\newblock {\em Bulletin of the London Mathematical Society}, 45(6):1259--1266,
  12 2013.

\bibitem{Plamondon2011}
P.-G. Plamondon.
\newblock {Cluster algebras via cluster categories with infinite-dimensional
  morphism spaces}.
\newblock {\em Compositio Mathematica}, 147(06):1921--1954, 11 2011.

\bibitem{Plamondon2013}
P.~G. Plamondon.
\newblock {Generic bases for cluster algebras from the cluster category}.
\newblock {\em International Mathematics Research Notices},
  2013(10):2368--2420, 2013.

\bibitem{Rudakov1997}
A.~Rudakov.
\newblock {Stability for an Abelian category}.
\newblock {\em Journal of Algebra}, 197(1):231--245, 1997.

\end{thebibliography}
\bibliographystyle{abbrv}

\end{document}